\documentclass[12pt] {amsart}
\usepackage{comment}
\usepackage{enumerate}
\usepackage{amssymb, amsmath}

\newtheorem{thm}{Theorem}[section]
\newtheorem{prop}[thm]{Proposition}
\newtheorem{lemma}[thm]{Lemma}
\newtheorem{cor}[thm]{Corollary}

\numberwithin{equation}{section}

\usepackage{url}

\def\Z{{\mathbb Z}}

\usepackage{hyperref}
\usepackage{geometry}
\usepackage[T1]{fontenc}
\usepackage{hyperref}
\usepackage{longtable}
\usepackage{enumerate}

 \hypersetup{
   linkcolor=green,
    filecolor=black,      
    urlcolor= cyan,
    }

\begin{document}

\title{A classification of Multiply Monogenic Quartic Orders}

\author{Shabnam Akhtari}
\address{Penn State University Mathematics Dept., McAllister Building, 
University Park, State College, PA 16802}

 \email {akhtari@psu.edu}
 
 \author{Jaxon Shumaker}
\address{Tillamook Bay Community College, Student Learning and Support Building, Tillamook, OR 97141}
 
 \email{jaxonshumaker@tillamookbaycc.edu}

\begin{abstract}
We study two-times monogenic quartic orders; i.e., those of the shape $\mathbb{Z}[\alpha] = \mathbb{Z}[\beta]$, with algebraic integers $\alpha$ and $\beta$  not $\mathbb{Z}$-equivalent. Two specific types, describing possible algebraic relation among monogenizers of two-times monogenic orders were defined by B\'erczes, Evetrse, Gy\H{o}ry in \cite{BEG}, where it is proven under certain conditions on the Galois group of the normal closure of a given number field $K$, that there can be only finitely many two-times monogenic $\mathbb{Z}$-orders in the ring of integers $K$ which
are not of these specific two types. In this article, we prove this fact for all quartic number fields. 
   \end{abstract}


\maketitle

\section{Introduction}

Let $K$ be an algebraic number field and $O_K$ its ring of integers, and $O$  an order in  $O_K$ (a subring of $O_K$ with quotient field $K$). We call the ring $O$ {\it monogenic} if it is generated by one element as a $\Z$-algebra, i.e., $O=\Z[\alpha]$ for some $\alpha\in O$; the element $\alpha$ is  called a {\it monogenizer} of $O$.  If $\alpha$ is a monogenizer of $O$, than so are $\pm \alpha+c$ for every $c\in\Z$. Two monogenizers $\alpha$ and $\alpha'$ of $O$  are called  {\it $\mathbb{Z}$-equivalent} (or {\it equivalent}) if $\alpha'=\pm \alpha +c$ for some $c\in\Z$.  
We denote an equivalence class of monogenizers of $O$  a  {\it monogenization} of $O$.

By the fundamental work of Gy\H{o}ry \cite{Gyo76}, we know that any order in an algebraic number field can have at most finitely many monogenizations and that effectively computable upper bounds on the number of these monogenizations can be determined.  
An overview of various results on estimates for the number of monogenizations of orders in number fields is given in  \cite{Eve11}. There are further extensions and generalizations of such results in  \cite{EGbook}  (in particular, see Section 9.1). 

An order $O$ is called \emph{$k$-times monogenic} if  there exist at least $k$ pairwise inequivalent elements 
$\alpha_{1}, \ldots, \alpha_{k} \in {O}$ such that 
\begin{equation}\label{ktimesdef}
{O} = \mathbb{Z}[\alpha_{1}] = \cdots = \mathbb{Z}[\alpha_{k}].
\end{equation}
Similarly, the order ${O}$ is called \emph{precisely $k$-times monogenic} 
 if there are precisely 
$k$ 
pairwise $\mathbb{Z}$-inequivalent elements 
$\alpha_{1}, \ldots, \alpha_{k} \in{O}$ such that  \eqref{ktimesdef} holds.

It is well-known that any order in a quadratic number field is precisely one-time monogenic (see \cite{GaussI}, for example). For number fields of larger degree,  the following remarkable theorem is established in \cite{BEG},  by utilizing finiteness results on the number of solutions of unit equations in
more than two unknowns, together with some combinatorial arguments.
\begin{thm}[B\'erczes, Evetrse, Gy\H{o}ry] \label{BEGthm1.1}
Let $K$ be a number field of degree greater than $2$. Then there are at most
finitely many three-times monogenic orders in $K$.
\end{thm}

Following  \cite{BEG},  we call  a monogenic order ${O}$ in $O_K$ 
 of \emph{type~I} if there exist $\alpha, \beta \in {O}$ and 
\[
C=\begin{pmatrix}
m_{1} & m_{2} \\
m_{3} & m_{4}
\end{pmatrix} \in \mathrm{GL}_2(\mathbb{Z})
\]
such that 
\begin{equation}\label{eq:typeI}
K = \mathbb{Q}(\alpha), \qquad 
{O} = \mathbb{Z}[\alpha] = \mathbb{Z}[\beta], \qquad 
\beta = \frac{m_{1}\alpha + m_{2}}{m_{3}\alpha + m_{4}}, \qquad 
m_{3} \neq 0 .
\end{equation}
Since $m_{3} \neq 0$,  $\beta$ is not  $\mathbb{Z}$-equivalent to $\alpha$

In \cite{BEG}, an order ${O}$  is called of \emph{type~II},
if there exist $\alpha, \beta \in {O}$ and 
$\mu_{0},\mu_{1},\mu_{2},\nu_{0},\nu_{1},\nu_{2} \in \mathbb{Z}$ with $\mu_{0}\nu_{0} \neq 0$ such that 
\begin{eqnarray}\label{eq:typeII}\nonumber
&& K = \mathbb{Q}(\alpha), \qquad 
{O} = \mathbb{Z}[\alpha] = \mathbb{Z}[\beta], \\
&&\beta = \mu_{0}\alpha^{2} + \mu_{1}\alpha + \mu_{2}, \qquad
\alpha = \nu_{0}\beta^{2} + \nu_{1}\beta + \nu_{2}.
\end{eqnarray}
Orders ${O}$ of type~II exist only for number fields of degree $4$.  Examples of number fields having infinitely many orders of type~I, 
and  of type~II are given in \cite{BEG}.

This article is motivated by results and conjectures in  \cite{BEG}, in which, among many other striking properties of monogenic orders,  the following theorem is proven for quartic number fields.
\begin{thm}[B\'erczes, Evetrse, Gy\H{o}ry]\label{BEGS4}
Let $K$ be a quartic number field whose normal closure has Galois group $S_{4}$. 
Then there are at most finitely many two-times monogenic orders in $K$ 
which are not of type~I or of type~II.
\end{thm}

The following is our main result.
\begin{thm}\label{mainType}
Let $K$ be a quartic number field.
Then there are at most finitely many two-times monogenic orders in $K$ 
which are not of type~I or of type~II.
\end{thm}

We consider three cases to study the monogenizations of $\mathbb{Z}[\xi]$, where $\xi$ is an algebraic integer of degree $4$. These cases are determined by how the cubic resolvent form of $\xi$ (see \eqref{defofFu,v}, for the definition) splits over  $\mathbb{Q}$.
Using Proposition \ref{KWlem}, 
 We prove the following three lemmas, which together with Theorem \ref{BEGS4}  lead to Theorem \ref{mainType}.

\begin{lemma}\label{V4case}
Let $K$ be a quartic number field whose normal closure has Galois group  $V_4$, then every two-times monogenic order is of type I.
\end{lemma}

\begin{lemma}\label{newlemVC}
Let $K$ be a quartic number field whose normal closure has Galois group  $C_4$, or $D_8$.
Then there are at most finitely many two-times monogenic orders in $K$
which are not of type~I.
\end{lemma}

\begin{lemma}\label{newlemA4}
Let $K$ be a quartic number field whose normal closure has Galois group $A_{4}$. 
Then there are at most finitely many two-times monogenic orders in $K$ 
which are not of type~I or of type~II.
\end{lemma}


Applications of an explicit approach, which was implemented in \cite{AkhEss} for the purpose of establishing improved bounds on the number of monogenizations of a quartic order, enable us to study relationships among monogenizations through the resolution of some Thue equations. This point of view is helpful in our proofs of Lemmas \ref{V4case} and \ref{newlemVC}. To prove Lemma \ref{newlemA4}, we will slightly modify some powerful tools established in \cite{BEG}. The proof of Theorem \ref{BEGS4} in \cite{BEG} uses finiteness results about unit equations, such as those in \cite{EG85, EGST}, which are ineffective. This makes it impossible to give a quantitative bound on the number of possible exceptions in Theorem \ref{BEGS4}, and in our Lemma \ref{newlemA4}. On the other hand,  in Section \ref{irreducibleresolvent}, the explicit nature of our method for the proof of Lemma \ref{newlemVC} allows us to find any two-times monogenic order that is not of type~I in a quartic order with a given discriminant.

In what follows, we denote two non-equivalent monogenizers of a given quartic order sometimes by $\xi$ and $\alpha$, and other times by $\alpha$ and $\beta$. This is to recall corresponding facts from both \cite{AkhEss} and \cite{BEG} wherever methods from these different papers are used.

\section{Preliminaries:  Discriminants, Thue Equations, and $\textrm{GL}_2(\mathbb{Z})$-actions}


\subsection{Discriminants}
We recall some basic definitions.
Let $K$ be an algebraic number field of degree $n$. 
Let $\alpha_{1}, \ldots, \alpha_{n}$ be  a linearly independent set of $n$ elements of $K$, and  $\sigma_{1}, \ldots, \sigma_{n} : K \rightarrow \mathbb{C}$   the
embeddings of $K$ into $\mathbb{C}$. The discriminant  $D_{K/\mathbb{Q}}(\alpha_{1}, \ldots, \alpha_{n})$ of $(\alpha_{1}, \ldots, \alpha_{n})$ is defined as the square of the determinant of the $n \times n$ matrix $\left(\sigma_{i}(\alpha_{j})\right)$,
 where $i, j \in \{1, \ldots, n\}$.
  Let $\beta_1, \beta_2, \ldots, \beta_n$ be an integral basis for an order $O$ in $O_K$. The discriminant of $O$ is defined as
 $
D_{K/\mathbb{Q}}(\beta_{1}, \ldots, \beta_{n}),
$
 and is independent of the choice of the integral basis  $\beta_1, \beta_2, \ldots, \beta_n$.
The discriminant of a number field $K$ is the discriminant of its maximal order, the ring of integers $O_K$.

The index $I(\alpha)$ of an algebraic integer $\alpha$ of degree $n$ is defined to be the index of the module $\mathbb{Z}[\alpha]$ in ${O}_{K}$, where $K = \mathbb{Q}(\alpha)$.

Let $\mathbf{P}(T) \in \mathbb{Z}[T]$ be a polynomial of degree $n$ and leading coefficient
 $a \in \mathbb{Z}$.
The discriminant $\textrm{Disc}(\mathbf{P})$ of $\mathbf{P}(T)$ is  
$$
\textrm{Disc}(\mathbf{P}) = a^{2n -2} \prod_{i <j} (\gamma_i - \gamma_j)^2,
$$
where $\gamma_1, \ldots, \gamma_n \in \mathbb{C}$ are the roots of $\mathbf{P}(T)$.

The discriminant of an algebraic number is equal to the discriminant of its minimal polynomial defined over $\mathbb{Z}$.
We have 
$$
D(\alpha) = I(\alpha)^2 D_K.
$$
A trivial but important observation for our purpose is that if $\mathbb{Z}[\alpha] = \mathbb{Z}[\beta]$, then
\begin{equation}\label{indexequal}
D(\alpha) = D(\beta), \quad \textrm{and} \,  I(\alpha)=  \pm I(\beta).
\end{equation}

Let $F(U, V) \in \mathbb{Z}[U, V]$ be a binary form of degree $n$ that factors over $\mathbb{C}$ as
$$
\prod_{i=1}^{n} (\alpha_{i} U - \beta_{i} V).
$$
The discriminant $D(F)$ of $F$ is given by
\begin{equation}\label{defofdisc}
D(F) = \prod_{i<j} (\alpha_{i} \beta_{j} - \alpha_{j}\beta_{i})^2.
\end{equation}

We note that the discriminant of the polynomial $F(U, 1) \in \mathbb{Z}[U]$ is equal to the discriminant of the binary form $F(U, V)  \in \mathbb{Z}[U, V]$.

\subsection{Thue Equations}\label{ThueBounds}

Let $F(U , V) \in \mathbb{Z}[U , V]$ be a binary form of degree at least $3$. If $F(U, V)$ is irreducible over $\mathbb{Q}$, for any integer $m$, it is shown in \cite{Thu}  that the equation 
$$
F(U , V) = m
$$
has at most finitely many solutions in integers $U$, $V$. These equations are called \emph{Thue equations}.  

In \cite{AkhEss} and \cite{Bha-notes}, it is shown how to reduce the problem of counting the number of monogenizations of a quartic order to the problem of counting the number of solutions of a cubic Thue equation and a family of related quartic Thue equations. We shall use this reduction to further study algebraic relations among different monogenizations of a given monogenic quartic order.

\subsection{Matrix Actions on Binary Forms}

 Let $F(U, V) \in \mathbb{Z}[U, V]$ and
$
A = \left( \begin{array}{cc}
a & b \\
c & d \end{array} \right)
$ be a $2 \times 2$ matrix with integer entries. 
We define the binary form $$F_{A}(U, V) \in \mathbb{Z}[U, V]$$
 by
$$
F_{A}(U , V) = F(a U+ b V,  c U+ d V).
$$

From the definition  of discriminant in \eqref{defofdisc}, we 
observe that for any $2 \times 2$ matrix $A$ with integer entries
\begin{equation}\label{St6}
D(F_{A}) = (\textrm{det} A)^{n (n-1)} D(F).
\end{equation}

We say that two integral binary forms $F$ and $G$ are {\it equivalent} if $G = \pm F_{A}$ for some $A \in \textrm{GL}_{2}(\mathbb{Z})$. This is, in fact, an equivalence relation. Moreover, the discriminants of two equivalent forms are equal.

For  $A = \left( \begin{array}{cc}
a & b \\
c & d \end{array} \right)
 \in \textrm{GL}_{2}(\mathbb{Z})$, and any $(u, v) \in \mathbb{Z}^2$, we clearly have 
$$A^{-1} = \pm \left( \begin{array}{cc}
d& -b \\
-c & a \end{array} \right)
$$
and 
$$F_{A} ( du-bv , -cu+av) =  \pm F(u , v).
$$
Therefore, there is a one-to-one correspondence between the solutions of the Thue equation $F(U, V) = \pm 1$ and those of the Thue equation $F_A(U , V) = \pm 1$.

\section{Monogenizers and solutions of  Thue Equations}\label{redtoThue}

Let $K = \mathbb{Q}(\xi)$ be a quartic number field, and assume that $\xi$ is an algebraic integer with the minimal polynomial 
\begin{equation}\label{minpolydef}
\mathbf{P}(T) = T^4 + a_{1} T^3 + a_{2} T^2 + a_{3} T + a_{4} \in \mathbb{Z}[T].
\end{equation}
 The \emph{cubic resolvent form} of $\xi$, or $\mathbf{P}(T)$,   is defined as follows.
\begin{equation}\label{defofFu,v}
 F(U, V)  = U^3 - a_{2} U^2 V+ (a_{1} a_{3} - 4a_{4}) U V^2 + (4 a_{2} a_{4} - a_{3}^2 - a_{1}^2 a_{4}) V^3.
\end{equation} 
We will use the following well-known fact, which can be found in many books on Galois theory (see \cite{DF}, for instance).
\begin{lemma}\label{discremark}
Suppose  $\xi$ is an algebraic integer of degree $4$ and $F(U, V)$ is its cubic resolvent form. We denote the algebraic conjugates of $\xi$ by $\xi^{(1)}$, $\xi^{(2)}$, $\xi^{(3)}$, and $\xi^{(4)}$, and define
$$
\xi_{i, j, k,l} = \xi^{(i)} \xi^{(j)} + \xi^{(k)} \xi^{(l)}.
$$
Then
\begin{equation}\label{prod2}
F(U, V) = (U - \xi_{1, 2, 3, 4} V) (U - \xi_{1, 3, 2, 4} V) (U - \xi_{1, 4, 2, 3} V).
\end{equation}
Further,  the discriminant of the cubic form $F(U, V)$  is equal to the discriminant of the minimal polynomial of $\xi$, as well as the discriminant of the ring $\mathbb{Z}[\xi]$.
\end{lemma}

 Following  \cite{AkhEss} and \cite{GaPePo}, we define
\begin{eqnarray}\label{defofQ1}
&&\, \, \, \, \, \, \,  \mathbf{Q}_{1}(X , Y, Z) =  \\ \nonumber
&&X^2 - a_{1} XY + a_{2} Y^2 + (a_{1}^2 - 2 a_{2}) XZ 
+(a_{3} - a_{1} a_{2}) YZ + (- a_{1} a_{3} + a_{2}^2 + a_{4}) Z^2,
\end{eqnarray}
and
\begin{equation}\label{defofQ2}
\mathbf{Q}_{2}(X , Y, Z) = Y^2 - XZ - a_{1} YZ + a_{2} Z^2.
\end{equation}
The coefficients of the quadratic forms $\mathbf{Q}_{1}(X, Y,Z)$ and $\mathbf{Q}_{2}(X, Y, Z)$ are expressed in terms of the coefficients of $\mathbf{P}(T)$, the minimal polynomial of $\xi$.

We summarize some results from  Section 4  of \cite{AkhEss}, which builds upon an effective algorithm from
\cite{GaPePo} to treat general index form equations in quartic number fields.
Assume $\mathbb{Z}[\xi] = \mathbb{Z}[\alpha]$.
 By \eqref{indexequal}, the algebraic integers $\alpha$ and $\xi$, up to sign, have the same index.
 Further, since $\alpha \in  \mathbb{Z}[\xi]$, we have
\begin{equation}\label{noI0xyz}
 \alpha = a_{\alpha} + x \xi + y\xi^2 + z \xi^3,
\end{equation}
with $x, y, z \in \mathbb{Z}$.  Since  $\alpha$ and $ x \xi + y\xi^2 + z \xi^3$ are $\mathbb{Z}$-equivalent,  without loss of generality,  we may assume 
$$
 \alpha =  x \xi + y \xi^2 + z \xi^3.
 $$
The following is Lemma 4.1  of  \cite{AkhEss}. 
\begin{prop}\label{GPPmodified}
The algebraic integer $x\xi + y \xi^2 + z\xi^3$, with $x , y, z \in \mathbb{Z}$ is a monogenizer of $\mathbb{Z}[\xi]$ if and only if there is a solution $(u, v) \in \mathbb{Z}^2$ of the cubic Thue equation 
\begin{equation}\label{cubicThueeqmodified}
F(U, V) = \pm 1
\end{equation}
such that $(x , y, z)$ satisfies the system of quadratic ternary equations
\begin{equation}\label{twoquadraticsmodified}
\mathbf{Q}_{1}(X , Y, Z) = u , \, \, 
\mathbf{Q}_{2}(X , Y, Z) = v.
\end{equation}
\end{prop}

Since the form  $F(U, V)$ is monic,   i.e., the coefficient of the term $U^3$ equals $1$, 
The pairs $(\pm1, 0)$ satisfy the equation $F(U, V) = \pm 1$, with $F$ given in  \eqref{defofFu,v}.
 This corresponds to the system of equations 
\begin{eqnarray}\label{system=0,1}\nonumber
\mathbf{Q}_1 (X, Y, Z)& =& \pm 1\\ 
\mathbf{Q}_2 (X, Y, Z)& = & 0,
\end{eqnarray}
where the ternary quadratic forms $\mathbf{Q}_1$ and $\mathbf{Q}_2$ are defined in \eqref{defofQ1} and \eqref{defofQ2}. We define the  binary  quadratic forms
\begin{eqnarray}\label{XYZpq}\nonumber
X(P, Q)& = &P^2  - a_{1} PQ+ a_{2} Q^2, \\
Y(P, Q)&=& PQ, \\ \nonumber
 Z(P, Q)&=&Q^2.
\end{eqnarray}
The integral  binary quadratic  forms  $X(P,Q)$, $Y(P,Q)$ and $Z(P,Q)$
 are constructed   by eliminating a variable in the ternary equation
$\mathbf{Q}_2 (X, Y, Z) =  0$. For more details on construction, we refer the reader to \cite[\S 4]{AkhEss},  in particular equations (27)--(30).

In \cite{AkhEss},  it is shown that there is a one-to-one correspondence between the set of  solutions $(X, Y, Z) \in \mathbb{Z}^3$ of the system of ternary equations \eqref{system=0,1} and the set of solutions $(P, Q) \in \mathbb{Z}^2$ of the quartic Thue equation 
\begin{equation}\label{Q10=1}
\mathcal{Q}_{(1, 0)}(P, Q)  = \pm 1,
\end{equation}
where
\begin{equation*}\label{letusquarticThue}
\mathcal{Q}_{(1, 0)}(P, Q) = \mathbf{Q}_{1}(X(P, Q) , Y(P,Q), Z(P,Q)).
\end{equation*}
More preciesly, each pair of solution $(p, q) \in \mathbb{Z}^2$  of  \eqref{Q10=1} corresponds to the monogenizer
$$
X(p, q) \xi + Y(p, q) \xi^2 + Z(p, q) \xi^3
$$
of the order $\mathbb{Z}[\xi]$. Clearly, the monogenizer $\xi$ is produced by the solution
 $(p, q) = ( 1, 0)$ of the quartic Thue equation \eqref{Q10=1}. We note that 
$$ \mathcal{Q}_{(1, 0)}(p, q) = \mathcal{Q}_{(1, 0)}(-p, -q)$$
for every $p, q \in \mathbb{Z}$. Further, since $X, Y, Z$ are quadratic forms, two integral solutions $(p, q)$ and $(-p, -q)$ of \eqref{Q10=1} produce the same monogenization of 
$\mathbb{Z}[\xi]$.

 \begin{lemma}\label{lemTypeI}
Let $K = \mathbb{Q}(\xi)$ be a quartic number field, and $\xi, \alpha \in O_K$.
Assume $\mathbb{Z}[\xi] = \mathbb{Z}[\alpha]$, and 
$
\alpha= x \xi + y \xi^2 + z \xi^3$, with
$ x, y, z \in \mathbb{Z}$.
Then there exist $a, b, c, d \in \mathbb{Z}$, with $ad - bc = \pm 1$ such that 
$$
\alpha = \frac{a\xi+ b}{c\xi + d},
$$
 if and only if  $(x , y, z) \in \mathbb{Z}^3 $ satisfies the system of equations \eqref{system=0,1}, where the ternary quadratic forms $\mathbf{Q}_1$ and $\mathbf{Q}_2$ are given in  \eqref{defofQ1} and \eqref{defofQ2}, with their coefficients expressed in terms of the coefficients of the minimal polynomial of  $\xi$.
\end{lemma}
\begin{proof}
Assume $\alpha$ is a monogenizer of $\mathbb{Z}[\xi]$ which satisfies
$$
\alpha = x \xi + y \xi^2 + z \xi^3 = \frac{a\xi+ b}{c\xi + d},
$$
with $x, y, z, a, b, c, d \in \mathbb{Z}$ and $c \neq 0$.  Then we have
$$
a\xi+ b = cz\xi^4 +(dz+cy)\xi^3 +(dy+cx)\xi^2 +xd\xi,
$$
and
$$
cz\xi^4 +(dz+cy)\xi^3 +(dy+cx)\xi^2 +(xd-a) \xi -b = 0.
$$
Since $\xi$ is an algebraic integer of degree $4$, we conclude that $z \neq 0$ and the coefficients of the  minimal polynomial  $\mathbf{P}(T)$, given in 
\eqref{minpolydef}, of $\xi$  are
\begin{equation}\label{aitypeI}
a_1 = \frac{dz+cy}{cz}, \quad a_{2} = \frac{dy+cx}{cz}, \quad a_{3} = \frac{dx- a}{cz}, \quad a_{4} = \frac{-b}{cz}.
\end{equation}

First, assume $y = 0$.  We readily get
$$
a_1 = \frac{d}{c} \quad \textrm{and} \quad  a_2 = \frac{x}{z}.
$$
In this case, we clearly have $(X, Y, Z) = (x, 0, z)$ satisfies the system of equations \eqref{system=0,1}.

Next, suppose $y \neq 0$. By \eqref{aitypeI}, we have 
$$
d = \frac{cza_2 - cx}{y} = \frac{cza_1 - cy}{z}.
$$
Since $c \neq 0$, we have
$$
 \frac{za_2 - x}{y} = \frac{za_1 - y}{z},
 $$
 that is $\mathbf{Q}_2(x , y, z) =0$,
 which means that $(x, y, z)$ satisfies the system of equations  \eqref{system=0,1}.
 
 We conclude that if the monogenizers $\xi$ and $\alpha$ of $\mathbb{Z}[\xi]$  have type~I relation and
 $$
 \alpha = x \xi + y \xi^2 + z \xi^3+c
 $$
then $(x, y, z)$ satisfy \eqref{system=0,1}. 

Now assume that 
$\alpha = x \xi + y \xi^2 + z \xi^3$ is a monogenizer of $\mathbb{Z}[\xi]$ and $(x, y, z) \in \mathbb{Z}^3$ satisfies  the system of equations  \eqref{system=0,1}. Assuming that the monogenizer $\alpha$  is not $\mathbb{Z}$-equivalent to $\xi$, \eqref{system=0,1} and \eqref{XYZpq} immediately imply that 
   $z \neq 0$. We proceed to show that $\xi$ and $\alpha$ have type~I relation. By the parametrization given in \eqref{XYZpq}, there exists a non-zero integer $q$ such that 
   \begin{equation}\label{zq2}
   z = q^2, \quad \textrm{and} \quad q\mid y.
   \end{equation}
  We take 
  \begin{equation}\label{letustakec}
  c = z,\, \quad d = \frac{a_1 z - y}{q}.
    \end{equation} 
Conseqently, since $\mathbf{Q}_2(x, y, z) = y^2 - xz - a_{1} yz + a_{2} z^2 = 0$, we have 
\begin{equation}\label{dyis}
dy = \left(a_1 z - y\right) y = a_2 z^2 - zx.
\end{equation}

Let us take
\begin{equation}\label{letustakea,b}
b =  -a_4 z , \,  \quad a = \frac{-a_3 z^2 + dx}{q} = \frac{-a_3 z^2 +a_1 x z - xy}{q} ,  
\end{equation}
where the last identity is by replacing $d$ with its value given in \eqref{letustakec}.
With the integer values of $a, b, c$,  $d$, and $q$ given in \eqref{letustakec}, \eqref{letustakea,b},  and \eqref{zq2}, we have
\begin{eqnarray*}
ad - bc &=& \frac{\left(-a_3 z^2 +a_1 x z - xy\right) \left(a_1 z - y\right)}{q^2} -  \left(-a_4 z^2 \right) \\
&=&  \frac{xy^2  +a_3 yz^2+{a_{1}}^2 x  z^2 -2a_1 xyz  +\left(-a_1 a_3 +a_4  \right)z^3}{z}\\
& = & \frac{z \mathbf{Q}_{1}(x, y, z) -a_1 xyz + xy^2 -x^2 z -a_2y^2z +2a_2 x z^2 +a_1 a_2 yz^2 -{a_2}^2 z^3}{z}\\
& = &  \mathbf{Q}_{1}(x, y, z)   +  \frac{\left(x -a_2 z\right)\left(y^2 - xz- a_{1} yz + a_{2} z^2\right)}{z}\\
& = &  \mathbf{Q}_{1}(x, y, z)   +  \frac{\left(x -a_2 z\right) \mathbf{Q}_{2}(x, y, z)}{z} = \pm 1,
\end{eqnarray*}
where the last idenity is because   $(x, y, z)$ satisfies \eqref{system=0,1}.
Therefore, $\left(\begin{array}{ll}
a & b\\
c & d
\end{array}\right) \in \textrm{GL}_{2}(\mathbb{Z})$. 
\end{proof}

Lemma \ref{lemTypeI}, together with Proposition \ref{GPPmodified}, implies the following.

\begin{cor}\label{corTypeI}
Let $K = \mathbb{Q}(\xi) = \mathbb{Q}(\xi)$ be a quartic number field, and $\xi, \alpha \in O_K$.
Assume $\mathbb{Z}[\xi] = \mathbb{Z}[\alpha]$. 
If $\xi$ and $\alpha$ are neither $\mathbb{Z}$-equivalent, nor have a type~I relation (see \eqref{eq:typeI} for definition), then there exists a non-trivial integral solution $(u, v) \neq (1, 0)$  to the cubic Thue equation $F(U, V) = 1$, where  $F(U, V)$ is the cubic resolvent form of $\xi$, given in \eqref{defofFu,v}.
 \end{cor}

\section{Reducible Cubic Resolvent Forms, Proof of Lemmas \ref{V4case}  and \ref{newlemVC}}\label{irreducibleresolvent}

The following is a summary of Theorem 1 of \cite{KW89}, which helps us break the proof of our main theorem into three cases.
\begin{prop}[Kappe and Warren]\label{KWlem}
Let  $K$ be a quartic number field, and $\tilde{K}$  the normal closure of $K$. Assume $\xi \in O_K$ and $K= \mathbb{Q}(\xi)$. Let $F(U, V) \in \mathbb{Z}[U, V]$ be the cubic resolvent form \eqref{defofFu,v} of the minimal polynomial of $\xi$, given in \eqref{minpolydef}. Then 
\begin{enumerate}[(i)]
    \item  $F(U, V)$ has three distict linear factors in $\mathbb{Z}[U, V]$,  if the Galois group of $\tilde{K}$ is isomorphic to  $V_4$,
\item $F(u, v)$ has a linear and an irreducible quadratic factor in $\mathbb{Z}[U, V]$, if the Galois group of $\tilde{K}$ is  isomorphic to $C_4$ or $D_8$, and 
\item $F(u, v)$ is irreducible in $\mathbb{Z}[U, V]$, if the Galois group of $\tilde{K}$ is  isomorphic to $A_4$ or $S_4$.
\end{enumerate}
\end{prop}

In the proof of Lemma \ref{newlemVC}, we will follow some arguments in Bennett's work \cite{Ben} on the number of solutions of cubic Thue equations. In particular, the following simple observation is extremely useful.

\begin{lemma}[\cite{Ben}, Lemma 2.2]\label{simplechangevar}
If $F$ is a reducible binary cubic form, then equation $F(U, V) = 1$ has a solution in integers $U$ and $V$ precisely when $F$ is equivalent to a
form $G$ satisfying
\begin{equation}\label{eq2.1bennett}
G(U, V) = U(U^2 + aUV + bV^2),
\end{equation}
where $a, b \in \mathbb{Z}$.
\end{lemma}

\subsection*{Proof of Lemma \ref{V4case}.}

Suppose $K$ is a $V_4$  quartic number field. We show that any two-times monogenic order in $O_K$ is of type~I. Let $\xi$ be an algebraic integer such that $K = \mathbb{Q}(\xi)$.
 By part (i) of Proposition \ref{KWlem}, the resolvent cubic forms of $\xi$ is  of the shape
 $$ 
F(U, V) = (U- z_1V) (U- z_2 V) (U - z_3 V), \,  \textrm{with}\,  z_1, z_2, z_3 \in \mathbb{Z}.
$$
 Since the discriminant of $F(U, V)$ is not zero, by Lemma  \ref{discremark},  the three lines 
 $$U- z_i V =1, \quad i \in \{ 1, 2, 3\},
 $$
  are distinct and cannot have more than one intersection. Therefore, the cubic equation
$$
F(U, V) = 1
$$
has exactly one solution, namely $(1, 0)$.  By  Proposition \ref{GPPmodified} and Corollary \ref{corTypeI}, if $\mathbb{Z}[\xi$] is two-times monogenic. then it must be of type~I.
\qed

\subsection*{Proof of Lemma \ref{newlemVC}.}

Now we consider the case where $K$ is either a $C_4$ or a $D_8$ quartic number field.
 Suppose $O= \mathbb{Z}[\xi]$ is a two-times monogenic order in $O_K$ that is not of type~I.  By Proposition \ref{KWlem} and Lemma \ref{simplechangevar}, the cubic resolvent  form 
$F(U, V) $ of the quartic algebraic integer $\xi$ is equivalent to 
$$G(U, V) = U(U^2 + aUV + bV^2),
$$
for some  $a, b \in \mathbb{Z}$, so that the quadratic form $U^2 + aUV + bV^2$ is irreducuble over $\mathbb{Q}$.
We have 
\begin{equation}\label{eq2.2bennett}
D(F) =  D(G) = b^2 (a^2 - 4b),
\end{equation}
where $D(F)$ is  the discriminant of $F(U, V)$.
Since, by Lemma  \ref{discremark},  $D(F) \neq 0$, we have $b \neq 0$.
If   $O= \mathbb{Z}[\xi]$ is a two-times monogenic order in $O_K$ that is not of type~I, then  by Proposition \ref{GPPmodified} and Corollary \ref{corTypeI},  the equation 
\begin{equation}\label{G=1}
G(U, V) = U(U^2 + aUV + bV^2) = 1
\end{equation}
must have an integral  solution $(u, v) \neq (1, 0)$.
We note that for $\xi \in O_K$, with $K = \mathbb{Q}(\xi)$, the roots of the cubic resolvent polynomial of $\xi$ belong to $\tilde{K}$, the Galois closure of $K$ (see \eqref{prod2}).
By the fundamental theorem of Galois theory, we know that when the Galois group of $\tilde{K}$ is isomorphic to  $C_4$, there is exactly one quadratic subfield of $\tilde{K}$. Further, when the Galois group of $\tilde{K}$ is isomorphic to  $D_8$,  there are $3$  quadratic subfields of $\tilde{K}$ (as the dihedral group $D_8$ has three subgroups of order $4$). 
Now assume that $L$ is a quadratic subfield of $\tilde{K}$, with fundamental discriminant $\Delta$. 
Next, we show that there can be at most finitely many (indeed very few, if any) quadratic forms $U^2+ a UV + bV^2$ that split in the quadratic field $L$ and for which the equation \eqref{G=1} has a non-trivial solution.

We apply some ideas from \cite[\S 2]{Ben}.
If $G(u,v) = 1$, with $u, v \in \mathbb{Z}$, then  $u = \pm 1$. 
In case $u = 1$, substituting into \eqref{G=1}, we find that
\begin{equation}\label{eq2.3bennett}
    b v^{2} + a v = 0, \quad \textrm{and} \quad v =-\frac{a}{b}.
\end{equation}
If $u = -1$, we get \begin{equation}\label{eq2.4bennett}
    b v^{2} + a v + 2 = 0 \quad \textrm{and} \quad  v = \frac{-a + \sqrt{a^{2}-8b}}{2b},\quad \textrm{or} \quad 
v = \frac{-a - \sqrt{a^{2}-8b}}{2b}.
\end{equation}

Assuming the equation \eqref{G=1} has an integral solution $(u, v) \neq (1, 0)$, 
we immediately conclude that the unit group of the ring of integers of  $L = \mathbb{Q}(\sqrt{\Delta})$ is not finite. Therefore, $\Delta  > 0$ and $\mathbb{Q}(\sqrt{\Delta})$  is a totally real subfield of $\tilde{K}$ whose ring of integers has a  unit group generated by a unique fundamental unit $\eta$, with 
$\eta> 1$. 
Let
\begin{equation}\label{fundamentalunit}
\eta =  e_1+ e_2\sqrt{\Delta} > 1, \, \,  \textrm{with}\,  e_1, e_2 \in  \frac{1}{2} \mathbb{Z}.
\end{equation}
We define
\begin{equation}\label{fundamentalunittilde}
\tilde{\eta} =  e_1 - e_2\sqrt{\Delta}.
\end{equation}
We have $$\tilde{\eta} = \pm \frac{1}{\eta},$$
 and 
 every unit with absolute value greater than $1$ in the ring of integers of $\mathbb{Q}(\sqrt{\Delta})$ is of the form $\pm \eta^n$, for some $n \geq 1$.
 Clearly,
 \begin{equation}\label{clearlimits}
 \lim_{n \to \infty} \left| \eta^n \right| = + \infty
  \quad \textrm{and} \quad  \lim_{n \to \infty}\left| {\tilde{\eta}}^n \right| = 0.
 \end{equation}
 Further, the identity $G(u, v) = 1$, for $u= \pm 1$, and $v \neq 0$ implies that 
$$
\left( u+v \frac{ a+ \sqrt{a^2-4b}}{2}\right) \left( u +v \frac{ a -\sqrt{a^2-4b}}{2}\right) = u.
$$
 This means $ u+v \frac{ a+ \sqrt{a^2-4b}}{2}$  and $u +v \frac{ a -\sqrt{a^2-4b}}{2}$ are non-trivial units in $\mathbb{Q}(\sqrt{\Delta})$, and 
 $$
\left\{  u+v \frac{ a+ \sqrt{a^2-4b}}{2}, \,  u +v \frac{ a -\sqrt{a^2-4b}}{2}\right\} = \{\eta^h, \tilde{\eta}^h\},
 $$
 or
 $$
\left\{  u+v \frac{ a+ \sqrt{a^2-4b}}{2}, \,  u +v \frac{ a -\sqrt{a^2-4b}}{2}\right\} = \{-\eta^h, -\tilde{\eta}^h\},
 $$
  for some positive integer $h$. Therefore,
$$
u + \frac{v a}{2} = \pm \frac{\eta^{h} + \tilde{\eta}^{h}}{2}, \, \quad  \frac{ v\sqrt{a^2-4b}}{2} = \pm \frac{\eta^{h} - \tilde{\eta}^{h}}{2},
$$
that is
\begin{eqnarray}\label{thismeans}\nonumber
u &= &\pm \frac{\eta^{h} + \tilde{\eta}^{h}}{2} -  \frac{v a}{2} = \pm \left( \frac{\eta^{h} + \tilde{\eta}^{h}}{2} -  a \frac{\eta^{h} - \tilde{\eta}^{h}}{2\sqrt{a^2-4b}} \right) \\
&=& \pm \frac{1}{2} \left( \eta^{h} \left(1 - \frac{a}{\sqrt{a^2-4b}} \right) + \tilde{\eta}^{h}\left(1 +\frac{a}{\sqrt{a^2-4b}} \right)  \right).
\end{eqnarray}
In the remainder of the proof, we show, in a fixed quadratic field $L$, that \eqref{thismeans}  can hold for at most finitely many integer values  $a, b$. To this end, we proceed to show $1 - \frac{a}{\sqrt{a^2-4b}}$ and $1 +\frac{a}{\sqrt{a^2-4b}}$ are bounded by explicit constants (see \eqref{bound1} and \eqref{bound2}), and since $u = \pm 1$, the identity \eqref{thismeans}, together with \eqref{clearlimits}, implies that the positive integer $h$ cannot be arbitrarily large.

First assume \eqref{eq2.3bennett} holds, and $(u, v) = (1, -\frac{a}{b})$. We have $|b| \leq |a|$ and
$$
a^2 \leq 2 (a^2 -4b) = a^2 + a^2 -8b, \quad \textrm{and} \quad 3 a^2 \geq 2(a^2 -4b),
$$
provided $\left| a \right| \geq 8$. 
This means 
\begin{equation}\label{bound1}
\frac{3}{2} \leq  \frac{a^2}{a^2-4b} \leq 2.
\end{equation}
Clearly, there are only finitely many pairs $(a, b)$, with $b|a$, and $|a| < 8$.

Now suppose \eqref{eq2.4bennett} holds. We may assume without loss of generality that $a \geq 0$, as the quadratic forms 
$$U^2 + a UV+ b V^2  \, \quad \textrm{and}\, \quad U^2  -a UV+ b V^2
$$
 are $\textrm{GL}_2(\mathbb{Z})$-equivalent. 
If $(u, v )=  \left(-1, \frac{-a + \sqrt{a^{2}-8b}}{2b}\right)$, then  $-a - \sqrt{a^{2}-8b }\in \mathbb{Z}$, and  via the straightforward identity
\begin{equation}\label{vias}
\frac{(-a + \sqrt{a^{2}-8b})( -a - \sqrt{a^{2}-8b} )}{2b} = 4,
\end{equation}
we conclude that
\begin{equation}\label{divides4}
-a - \sqrt{a^{2}-8b }\mid 4.
\end{equation}
One can easily verify that there are only finitely many integer values of $a \geq 0$ and  $b \neq 0$ for which \eqref{divides4} holds. They are
$$
(a , b) \in \{ (0, -2),  (1, -1), (3, 1)\}.
$$

If  $(u, v )=  \left(-1,  \frac{-a - \sqrt{a^{2}-8b}}{2b}\right)$, then we take
 $$\kappa = -a +\sqrt{a^{2}-8b } \in \mathbb{Z}.
 $$ 
 By \eqref{vias}, we have 
$$\kappa \mid 4 \quad \textrm{and} \quad \frac{-2a -\kappa }{2b}  = \frac{-a - \sqrt{a^{2}-8b}}{2b} =  \frac{4}{\kappa}.
$$
Therefore, 
\begin{equation}\label{a,b,kap}
a^2 - 4b = (a+\kappa)^2 + 4b = (a+\kappa)^2 +  \frac{\kappa (-2a -\kappa)}{2}  = a^2 +a \kappa +\frac{\kappa^2}{2}.
\end{equation}
Note that if  $a = 0$, then $b = -2$, and if $a=1$ then $b= -1$. For $a\geq 2$, since $\kappa \mid 4$,  we have
\begin{equation}\label{5a2}
5a^2 \geq  a^2 +a \kappa +\frac{\kappa^2}{2}.
\end{equation}
Further, the inequality
$$
2\left( a^2 +a \kappa +\frac{\kappa^2}{2} \right) = a^2 + (a+\kappa)^2 \geq a^2,
$$
together with \eqref{a,b,kap} and 
\eqref{5a2}, implies that
\begin{equation}\label{bound2}
\frac{1}{5} \leq \frac{a^2}{a^2 - 4b} \leq 2.
 \end{equation}

\qed


\section {Auxiliary Results; Applications of Finiteness Theorems for Unit Equations}

To complete the proof of Theorem~\ref{mainType}, we will use the machinery developed in  \cite{BEG}. In Section \ref{A4Section},
we will slightly modify some of the results in \cite{BEG}, which were proven for $S_4$-quartic number fields, and prove similar statements for $A_4$-quartic number fields.
First, we recall some useful definitions and general lemmas from  \cite{BEG}.

\begin{lemma}\cite[Lemma 6.1]{BEG}
\label{lem6.1BEG}
Let $\alpha, \beta \in O_{K}$ and suppose that 
$$\mathbb{Q}(\alpha)=\mathbb{Q}(\beta)=K, \quad \textrm{and} \quad \mathbb{Z}[\alpha]=\mathbb{Z}[\beta].
$$
 Then for distinct $ i,j \in \{1,2, 3, 4\}$, we have
\[
    \frac{\beta^{(i)} - \beta^{(j)}}{\alpha^{(i)} - \alpha^{(j)}} \in O^{\ast}_{\tilde{K}},
        \]
  where   $O^{\ast}_{\tilde{K}}$ is the unit group of the ring of integers of $\tilde{K}$.
\end{lemma}

Let ${O} = \mathbb{Z}[\alpha] =  \mathbb{Z}[\beta]$. For distinct $i, j \in \{1, 2, 3, 4\}$, we define
\begin{equation}\label{defofepsilon}
\varepsilon_{i j} =  \frac{\beta^{(i)} - \beta^{(j)}}{\alpha^{(i)} - \alpha^{(j)}}.
\end{equation}
We note that $\varepsilon_{ij} = \varepsilon_{ji}$, and by Lemma \ref{lem6.1BEG}, $\varepsilon_{i j} \in O_{\tilde{K}}^{\ast}$.
Moreover,   $\mathbb{Z}[\alpha] = \mathbb{Z}[\beta]$ implies that $I(\alpha) = \pm I(\beta)$, by \eqref{indexequal}. We obtain
\begin{equation}\label{epsilonijpm1}
\prod_{1\le  i< j \le  4} \varepsilon_{i j} = \pm 1.
\end{equation}
We also have
\begin{equation}\label{eq8.3BEG}
\frac{\varepsilon_{ik}}{\varepsilon_{jk}} = 
\left(\frac{\beta^{(i)}-\beta^{(j)} }
     {\beta^{(i)}-\beta^{(k)} }\right)
\left( \frac{\varepsilon_{ij}}{\varepsilon_{jk}} - 1 \right),
\end{equation}
 for any three distinct indices $i,j,k  \in \{1, 2, 3, 4\}$.

Two algebraic numbers  $\gamma_1$ and $\gamma_2$ are called  $\mathbb{Q}$-equivalent if 
$$
\gamma_{2} =   q_1\gamma_1 + q_2, \quad \textrm{with} \, q_1, q_2 \in \mathbb{Q}, \, q_1\neq 0.
$$
If $q_1, q_2 \in \mathbb{Z}$, with  $q_1\neq 0$, $\gamma_1$ and $\gamma_2$ are called  $\mathbb{Z}$-equivalent.
The next lemma demonstrates a simple relation that determines $\mathbb {Z}$-equivalence among the monogenizers of an order.
\begin{lemma}\cite[Lemma 6.3]{BEG}\label{lem6.3BEG}
For $\alpha \in K$, with $K =\mathbb{Q}(\alpha)$, we define the pair
\begin{equation}\label{eq6.1BEG}
\tau(\alpha) = \left(\frac{\alpha_3 - \alpha_1}{\alpha_2 - \alpha_1}, \frac{\alpha_4 - \alpha_1}{\alpha_2 - \alpha_1} \right).
\end{equation}
\begin{enumerate}[(i)]
    \item Let $\alpha, \beta$ with $\mathbb{Q}(\alpha)=\mathbb{Q}(\beta)=K$.  
    Then $\alpha, \beta$ are $\mathbb{Q}$-equivalent if and only if 
    \[
    \tau(\alpha)=\tau(\beta).
    \]
   \item Let $\alpha, \beta \in O_K$ and suppose that $K = \mathbb{Q}(\alpha) = \mathbb{Q}(\beta)$ and $\mathbb{Z}[\alpha] = \mathbb{Z}[\beta]$. Then $\alpha$ and $\beta$ are $\mathbb{Z}$-equivalent if and only if $\tau(\alpha) = \tau(\beta)$.
    \end{enumerate}
\end{lemma}

\begin{lemma}\cite[Lemma 6.4]{BEG}\label{lem6.4BEG}
Let $\alpha,\beta \in O_K$ with 
$$\mathbb{Q}(\alpha)=\mathbb{Q}(\beta)=K, \quad  \textrm{and} \quad \mathbb{Z}[\alpha]=\mathbb{Z}[\beta].
$$
Suppose there is a matrix
\[
\begin{pmatrix}
q_1 & q_2 \\
q_3 & q_4
\end{pmatrix} \in \textrm{GL}_2(\mathbb{Q})
\]
such that
\begin{equation*}\label{eq:6.3}
\beta=\frac{q_1\alpha + q_2}{q_3\alpha + q_4}, \qquad q_3 \neq 0,
\end{equation*}
and
\begin{equation*}\label{eq:6.4}
q_1,q_2,q_3,q_4 \in \mathbb{Z} , \qquad \gcd(q_1,q_2,q_3,q_4)=1.
\end{equation*}
Then
\[
\begin{pmatrix}
q_1 & q_2 \\
q_3 & q_4
\end{pmatrix} \in \textrm{GL}_2(\mathbb{Z}).
\]
\end{lemma}

\begin{lemma}\cite[Lemma 7.2]{BEG}\label{lem7.2BEG}
Let $\mathcal{C}$ be a $\mathbb{Q}$-equivalence class in $K$. Then the set of 
$\beta$ such that 
\[
\beta \in O_K \cap \mathcal{C},\qquad \mathbb{Q}(\beta)=K,\qquad \mathbb{Z}[\beta]\ \text{is two times monogenic}
\]
is contained in a union of at most finitely many $\mathbb{Z}$-equivalence classes.
\end{lemma}

We end this section with a non-trivial application of the theory of unit equations. The main tools in the highly technical proof of the following proposition are finiteness results on polynomial unit equations.

\begin{prop} \cite[Proposition 8.1]{BEG}\label{prop8.1BEG}
Let $O_{\tilde{K}}^{*}$ be the unit group of the number field $\tilde{K}$.
Consider  the  equation \begin{equation}\label{eq8.5BEG}
(x_1 - 1)(x_2 - 1)(x_3 - 1) = (y_1 - 1)(y_2 - 1)(y_3 - 1)
\end{equation}
in  variables $x_1, x_2, x_3, y_1, y_2, y_3$.
There is a finite subset $S \subseteq O_K^{*}$ with $1 \in S$ such that 
for every solution $(x_1,x_2,  x_3,y_1,y_2 ,y_3) \in {O_{\tilde{K}}^{*}}^{6}$  of  \eqref{eq8.5BEG}
at least one of the following holds.
\begin{enumerate}
    \item[(i)] 
    At least one of $x_1, x_2,  x_3, y_1, y_2, y_3$ belongs to $S$.
    
    \item[(ii)] 
    There exist $\eta_1, \eta_2, \eta_3 \in \{\pm 1\}$ such that 
    $(y_1, y_2, y_3)$ is a permutation of 
    \[
        \left( x_1^{\eta_1},\, x_2^{\eta_2},\, x_3^{\eta_3} \right).
    \]
    \item[(iii)] 
One of the numbers in 
\[
\{\, x_i x_j,\; x_i/x_j,\; y_i y_j,\; y_i/y_j : 1 \le i < j \le 3 \,\}
\]
is equal to either $-1$ or to a primitive third root of unity.
\end{enumerate}
\end{prop}


\section{$A_4$-quartic number fields}\label{A4Section}


We continue to assume that  $K = \mathbb{Q}(\alpha)$ is a 
quartic number field, and 
we denote by $\tilde{K}$ the normal closure of $K$ over $\mathbb{Q}$. Further, we will assume that 
$$G= \textrm{Gal}(\tilde{K}/\mathbb{Q}) \cong A_4.
$$
We denote the algebraic conjugates of $\alpha$ by $\alpha^{(1)}, \, \alpha^{(2)}, \, \alpha^{(3)},\, \alpha^{(4)}$.   Associated to each  $\sigma \in G$, there is a permutation in $A_4$, say $P_{\sigma}$,  
 viewed as a permutation of the
superscripts of $\alpha^{(1)}, \, \alpha^{(2)}, \, \alpha^{(3)},\, \alpha^{(4)}$. In what follows, sometimes by abuse of terminology, we speak of applying the $A_4$-permutation  $P_{\sigma}$ to an element of 
$\lambda \in \tilde{K}$ to mean applying
$\sigma$ to $\lambda$, that is, computing $\sigma(\lambda)$.

\begin{lemma}\label{shabnot1} Let $K$ be a quartic number field whose Galois closure has Galois group $A_4$.
Let $\alpha, \beta \in O_K$ and suppose that 
$$K = \mathbb{Q}(\alpha) = \mathbb{Q}(\beta)\, \quad \textrm{and}\, \quad \mathbb{Z}[\alpha] = \mathbb{Z}[\beta].
$$
Further, assume that $\alpha$ and $\beta$ are not $\mathbb{Z}$-equivalent. 
Let $i, j , k\in \{1, 2, 3, 4\}$ be distinct. We have $\varepsilon_{ki} \neq \varepsilon_{kj}$, where $\varepsilon_{ki}, \varepsilon_{kj}$ are defined in \eqref{defofepsilon} in terms of the algebraic conjugates of  $\alpha$ and $\beta$.
\end{lemma}
\begin{proof}
Assume  $\varepsilon_{1i} = \varepsilon_{1j}$, with distinct $i, j \in \{2, 3, 4\}$. Applying the $A_4$ permutations $(2\,  3\,  4)$ and $(2\,  4\,  3)$, we conclude that
$$
 \varepsilon_{12} = \varepsilon_{13} =  \varepsilon_{14}.
 $$
 This implies $\tau(\alpha) = \tau(\beta)$, which, by Lemma \ref{lem6.3BEG}, contradics our assumption of $\alpha$ and $\beta$  not being $\mathbb{Z}$-equivalent.
 More generally, if $\varepsilon_{ki} = \varepsilon_{kj}$, we may apply the two $A_4$ permutations that fix the index $k$ and permute the other three indices.  Similarly, we arrive at a contradiction with Lemma  \ref{lem6.3BEG}.
 \end{proof}

The next proposition consists of basic facts about $A_4$-quartic number fields, which will be used frequently in our proofs of the main results.
\begin{prop}\label{rep9.3BEG}
Let $K$ be a quartic number field whose normal closure $\tilde{K}$ has  Galois group isomorphic to $A_4$. 
\begin{enumerate}
   \item[(i)] 
    $\textrm{Gal}(\tilde{K}/K) \cong \{  (1), (2 \, 3\, 4), (2 \, 4\, 3)\}$.
     \item[(ii)]  The field $\tilde{K}$ does not contain any primitive second or third root of unity.
    \end{enumerate}
\end{prop}
\begin{proof}

To see part (i) holds, we simply write $K = \mathbb{Q}(\alpha)$, and denote the algebraic conjugates of $\alpha$ by  $\alpha^{(1)}$,   $\alpha^{(2)}$,  $\alpha^{(3)}$, and 
$\alpha^{(4)}$. Assuming $\alpha = \alpha^{(1)}$, it is clear that   the subgroup of $A_4$ consisting of those permutations of $1, 2, 3,4$   that fix $1$ is isomorphic to  
$\textrm{Gal}(\tilde{K}/K)$.

Since $A_4$ has no subgroup of order $6$ (see \cite{BrMa}, for example), by the fundamental theorem of Galois theory, we conclude that  $\tilde{K}$ has no quadratic subfield. This proves part (ii).
\end{proof}

We use Proposition \ref{rep9.3BEG}, to modify Lemma 9.3 of \cite{BEG} and prove a similar result in the case of $A_4$-quartic number fields.

\begin{lemma}\label{lem9.3BEG}
Let $K$ be a quartic number field whose normal closure has Galois group
 $$ G \cong A_4.
 $$ 
 There is a finite set $E$ such that for every  pair $(\alpha, \beta)$  satisfying   
$$
K = \mathbb{Q}(\alpha) = \mathbb{Q}(\beta), \quad \mathbb{Z}[\alpha] = \mathbb{Z}[\beta], $$
and $\alpha, \beta$ not $\mathbb{Z}$-equivalent, at least one of the following alternatives holds:
\begin{enumerate}
    \item[\textnormal{(i)}] 
    $\varepsilon_{ij}/\varepsilon_{ik} \in E$ for each ordered triple $(i,j,k)$ of distinct indices from $\{1,2, 3, 4\}$,

    \item[\textnormal{(ii)}] 
    $\varepsilon_{ij}\,\varepsilon_{kl} = \varepsilon_{ik}\,\varepsilon_{jl}$ for each permutation  $(i, j,k,l)$ of $(1, 2, 3, 4)$,

    \item[\textnormal{(iii)}] 
    $\varepsilon_{ij} =  -\varepsilon_{kl}$ for each permutation $(i,j,k,l)$ of $(1,2,3,4)$.
    
  \end{enumerate}

\end{lemma}

\begin{proof}
Our proof is similar to that of Lemma 9.3 of \cite{BEG}. We replace the set  $\mathcal{E}$ defined for the case of $S_4$-quartic number fields in \cite{BEG},  by a similar set $E$, defined below in \eqref{defofE}.
 Moreover, we will use the fact that the normal closure of an $A_4$-quartic number field contains no primitive second or third root of unity. 

We recall (see  equation \cite[eq. (8.4)]{BEG})
\begin{equation}\label{eq8.4BEG}
\left(  \frac{\varepsilon_{ik}}{\varepsilon_{jk}} - 1\right) \left(  \frac{\varepsilon_{il}}{\varepsilon_{kl}} - 1\right) \left(  \frac{\varepsilon_{ij}}{\varepsilon_{jl}} - 1\right) = 
\left(  \frac{\varepsilon_{ij}}{\varepsilon_{jk}} - 1\right) \left(  \frac{\varepsilon_{ik}}{\varepsilon_{kl}} - 1\right) \left(  \frac{\varepsilon_{il}}{\varepsilon_{jl}} - 1\right).
\end{equation}
which can easily be verified by the definition in \eqref{defofepsilon}.
Lemma \ref{lem6.1BEG}, together with \eqref{eq8.4BEG}, implies that
\begin{equation}\label{specialunitsol}
\left(  \frac{\varepsilon_{ik}}{\varepsilon_{jk}}, \frac{\varepsilon_{il}}{\varepsilon_{kl}} ,  \frac{\varepsilon_{ij}}{\varepsilon_{jl}},
\frac{\varepsilon_{ij}}{\varepsilon_{jk}},  \frac{\varepsilon_{ik}}{\varepsilon_{kl}} , \frac{\varepsilon_{il}}{\varepsilon_{jl}} \right) \in {O_{\tilde{K}}^{*}}^6
\end{equation}
satisfies the equation \eqref{eq8.5BEG}. We apply Proposition \ref{prop8.1BEG} to the specific solution \eqref{specialunitsol} of the equation \eqref{eq8.5BEG}.

Let $S$ be the finite set from Proposition \ref{prop8.1BEG}.
Let $E$ consist of all algebraic conjugates 
 of the elements from $S$, and their multiplicative inverses; i.e.,
\begin{equation}\label{defofE}
E =\{ \sigma(\gamma), \left(\sigma(\gamma)\right)^{-1} : \gamma \in S, \sigma \in G \}.
\end{equation}
First,  suppose for the solution \eqref{specialunitsol} that case~(i) of Proposition \ref{prop8.1BEG} holds.  
This, via the definition set $E$ in  \eqref{defofE},  leads to case~(i) of this lemma.

This is case~(i) of this lemma.

Next,  suppose that case~(ii) of Proposition \ref{prop8.1BEG} holds.  In the remainder of the proof, we work with the fixed choice of $(i, j, k, l) = (1, 2, 3,4)$, and the general result follows similarly. For this fixed choice, we have
\[
\frac{\varepsilon_{13}} {\varepsilon_{23}} \in \left\{  \frac{\varepsilon_{12}} {\varepsilon_{23}}, \frac{\varepsilon_{13}} {\varepsilon_{34}}, \frac{\varepsilon_{14}} {\varepsilon_{24}},    \frac{\varepsilon_{23}} {\varepsilon_{12}}, \frac{\varepsilon_{34}} {\varepsilon_{13}}, \frac{\varepsilon_{24}} {\varepsilon_{14}}\right\}.
\]
By Lemma \ref{shabnot1}, one can easily see that   $\frac{\varepsilon_{13}} {\varepsilon_{23}}$ cannot be equal to either  $\frac{\varepsilon_{12}} {\varepsilon_{23}}$ or $\frac{\varepsilon_{13}} {\varepsilon_{34}}$.
Next, we proceed to show that 
$ \frac{\varepsilon_{13}} {\varepsilon_{23}} \neq \frac{\varepsilon_{23}} {\varepsilon_{12}}.
$
 Suppose, on contarary,  $\frac{\varepsilon_{13}} {\varepsilon_{23}} = \frac{\varepsilon_{23}} {\varepsilon_{12}}$, then  
 $${\varepsilon_{23}}^2 =  {\varepsilon_{12}}{\varepsilon_{13}}.
 $$
 Applying the permutation $(1\, 2\,  3) \in A_4$, we get
$${\varepsilon_{31}}^2 =  {\varepsilon_{23}}{\varepsilon_{21}}.
$$
 The last two identities imply
$
 \left(\frac{\varepsilon_{23}}{\varepsilon_{31}} \right)^3 = 1
 $. By part (ii) of Proposition \ref{rep9.3BEG},
we conclude that $\frac{\varepsilon_{23}}{\varepsilon_{31}} =1$,
which is again impossible by Lemma \ref{shabnot1}. 
 Similarly, one can see
  $$ \frac{\varepsilon_{13}} {\varepsilon_{23}} \neq  \frac{\varepsilon_{34}} {\varepsilon_{13}}.
 $$

So far we have shown that, under the assumption that case~(ii) of Proposition \ref{prop8.1BEG} holds for our specific solution
$$
\left(  \frac{\varepsilon_{13}}{\varepsilon_{23}}, \frac{\varepsilon_{14}}{\varepsilon_{34}} ,  \frac{\varepsilon_{12}}{\varepsilon_{24}},
\frac{\varepsilon_{12}}{\varepsilon_{23}},  \frac{\varepsilon_{13}}{\varepsilon_{34}} , \frac{\varepsilon_{14}}{\varepsilon_{24}} \right)
$$
to  the equation \eqref{eq8.5BEG}, we must have
\[
\frac{\varepsilon_{13}} {\varepsilon_{23}} \in \left\{   \frac{\varepsilon_{14}} {\varepsilon_{24}},  \frac{\varepsilon_{24}} {\varepsilon_{14}}\right\}.
\]

Now assume  $\frac{\varepsilon_{13}} {\varepsilon_{23}} = \frac{\varepsilon_{14}} {\varepsilon_{24}}$.
Then we have 
$
{\varepsilon_{13}} {\varepsilon_{24}} = {\varepsilon_{23}} {\varepsilon_{14}}$, 
and applying $A_4$-permutations $(2\,  3 \, 4)$ and $(2\,   4 \, 3)$, we get
\[
{\varepsilon_{14}} {\varepsilon_{32}} = {\varepsilon_{34}} {\varepsilon_{12}} , \quad  \, {\varepsilon_{12}} {\varepsilon_{43}} = {\varepsilon_{42}} {\varepsilon_{13}},
\]
respectively. Therefore, we get case~(ii) of this lemma. 

If, on the the other hand, $\frac{\varepsilon_{13}} {\varepsilon_{23}} =  \frac{\varepsilon_{24}} {\varepsilon_{14}}$, we proceed to show that this is the case~(iii) of this lemma.
Assume
 $
\frac{\varepsilon_{13}} {\varepsilon_{23}} =  \frac{\varepsilon_{24}} {\varepsilon_{14}}$. 
Then
\begin{equation}\label{morework1}
{\varepsilon_{13}} {\varepsilon_{14}} = {\varepsilon_{23}} {\varepsilon_{24}}.
\end{equation}
Applying the $A_4$-permutation $(1\, 3)(2\,4)$ to both sides of \eqref{morework1},we get
\begin{equation}\label{morework2}
{\varepsilon_{13}} {\varepsilon_{32}} = {\varepsilon_{41}} {\varepsilon_{24}}
\end{equation}
Combining \eqref{morework1} and \eqref{morework2}, we obtain
\[
{\varepsilon_{14}}^2 = {\varepsilon_{23}}^2
\]
Since, by part (ii) of Proposition \ref{rep9.3BEG}, there is no primitive second root of unity in the number field $\tilde{K}$, we have
\[
{\varepsilon_{14}} = {\varepsilon_{23}} \quad \textrm{or} \quad
{\varepsilon_{14}} = - {\varepsilon_{23}}.
\]

Let us first assume that 
\begin{equation}\label{evenmore1}
{\varepsilon_{14}} = {\varepsilon_{23}}.
\end{equation}
Applying  the $A_4$ permutations   $(2\,  3 \, 4)$ and  $(2\,  4\,  3)$ to \eqref{evenmore1}, we get
\begin{equation}\label{evenmore2}
{\varepsilon_{12}} = {\varepsilon_{34}}, \quad {\varepsilon_{13}}= {\varepsilon_{42}},
\end{equation}
respectively.
Substitutig \eqref{evenmore1} and \eqref{evenmore2}  into \eqref{eq8.4BEG} with $(i,  j , k , l) = (1, 2,  3, 4)$, we have
\[
\left(\frac{\varepsilon_{13}} {\varepsilon_{14}} -1 \right) \left(\frac{\varepsilon_{14}} {\varepsilon_{12}} -1 \right) \left( \frac{\varepsilon_{12}} {\varepsilon_{13}}-1 \right)  =   \left(\frac{\varepsilon_{12}} {\varepsilon_{14}} -1 \right) \left( \frac{\varepsilon_{13}} {\varepsilon_{12}}-1 \right) \left( \frac{\varepsilon_{14}} {\varepsilon_{13}}-1 \right),
\]
 which is impossible as the left side is the negative of the right side, while by Lemma \ref{shabnot1}, both sides are non-zero.
Therefore, the identity ${\varepsilon_{14}}^2 = {\varepsilon_{23}}^2$
implies that
 \begin{equation}\label{negabove}
{\varepsilon_{14}} = -{\varepsilon_{23}}.
\end{equation}
Applying  the $A_4$-permutations $(2\,  3\,  4)$  and  $(2\,  4\,  3)$ to \eqref{negabove}, we get  
 \begin{equation}\label{negafterperm}
{\varepsilon_{12}} = - {\varepsilon_{34}}, \quad {\varepsilon_{13}}=  -{\varepsilon_{42}}.
\end{equation}
The identities in \eqref{negabove} and \eqref{negafterperm} lead to case~(iii) of this lemma.

Finally, we will show that case~(iii) of Proposition \ref{prop8.1BEG} cannot hold for our specific solution.
By part (i) of  Proposition \ref{rep9.3BEG}, there is no primitive third root of unity in $\tilde{K}$.
In what follows, we observe  that 
\[
\frac{\varepsilon_{ik}\,\varepsilon_{il}}{\varepsilon_{jk}\,\varepsilon_{kl}} \neq -1,
\quad
\frac{\varepsilon_{ik}\,\varepsilon_{kl}}{\varepsilon_{il}\,\varepsilon_{jk}} \neq -1.
\]
for any  permutation 
$(i,j,k,l)$ of $(1,2,3,4)$.
For simplicity, we show this for the fixed choice  $(i, j, k, l) = (1, 2, 3, 4)$, and the general claim can be shown similarly.

If $\frac{\varepsilon_{13}\,\varepsilon_{14}}{\varepsilon_{23}\,\varepsilon_{34}} = -1$, 
then
\[
\varepsilon_{13}\,\varepsilon_{14} =  -{\varepsilon_{23}\,\varepsilon_{34}},
\]
and applying the permutations $(3 4 2)$ and $(3 2 4)$, we get
\[
\varepsilon_{14}\,\varepsilon_{12} =  -{\varepsilon_{34}\,\varepsilon_{42}} ,
\]
\[
\varepsilon_{12}\,\varepsilon_{13} =  -{\varepsilon_{42}\,\varepsilon_{23}}.
\]
Consequently,
\begin{equation}\label{twoandtwo}
\left(\varepsilon_{12}\varepsilon_{13}\,\varepsilon_{14}\right)^2 =  -\left({\varepsilon_{24}\,\varepsilon_{23}\,\varepsilon_{34}} \right)^2.
\end{equation}
By \eqref{epsilonijpm1}, we have $\varepsilon_{12}\varepsilon_{13}\,\varepsilon_{14}\, {\varepsilon_{24}\,\varepsilon_{23}\,\varepsilon_{34}} = \pm 1$. Therefore, the identity \eqref{twoandtwo} implies
\[
\left(\varepsilon_{12}\varepsilon_{13}\,\varepsilon_{14}\right)^4 = -1,
\]
 which is impossible as $\tilde{K}$ does not contain any primitive second root of unity, by part (ii) of Proposition \ref{rep9.3BEG}.
 
 If $\frac{\varepsilon_{13}\,\varepsilon_{34}}{\varepsilon_{14}\,\varepsilon_{23}} = -1$, 
 then
  \[
 {\varepsilon_{13}\,\varepsilon_{34}} = - {\varepsilon_{14}\,\varepsilon_{23}} ,
 \]
and  applying the permutations $(3\,  4 \, 2)$ and $(3\,  2 \, 4)$, we get
 \[
 {\varepsilon_{14}\,\varepsilon_{42}} = - {\varepsilon_{12}\,\varepsilon_{34}} ,
 \]
 \[
 {\varepsilon_{12}\,\varepsilon_{23}} = - {\varepsilon_{13}\,\varepsilon_{42}} .
 \]
 Consequently,
 \[
 \varepsilon_{13}\,\varepsilon_{34} \, \varepsilon_{14}\,\varepsilon_{42} \, \varepsilon_{12}\,\varepsilon_{23} = - \varepsilon_{14}\,\varepsilon_{23} \, \varepsilon_{12}\,\varepsilon_{34}\,  {\varepsilon_{13}\,\varepsilon_{42}},
 \]
 which is impossible, as the left-hand and the right-hand sides of the above identity are non-zero, by \eqref{epsilonijpm1}.
  \end{proof}

 \section{Proof of Lemma \ref{newlemA4}}\label{lastproof}

We will closely follow the proof of part (ii) of Theorem 1.2 in \cite{BEG}, which is for the case of $S_4$-quartic number fields. All the required adjustments for the new case of $A_4$-quartic number fields are established in the previous lemmas in Section \ref{A4Section}.

\begin{proof}
Let  
\begin{equation}\label{thepaialphabeta}
{O} = \mathbb{Z}[\alpha] = \mathbb{Z}[\beta]
\end{equation}
be a two-times monogenic order in $O_K$, where $\alpha, \beta$ are not $\mathbb{Z}$-equivalent.

{\it Case~(i)}.  Assume that the pair $(\alpha,\beta)$  in \eqref{thepaialphabeta}
satisfies alternative~(i) of Lemma~\ref{lem9.3BEG}.
By Lemma~\ref{shabnot1} and \eqref{eq8.3BEG}, there is a finite set $F$ independent of 
$\alpha, \beta$ such that 
\[
\frac{\beta^{(i)} - \beta^{(j)}}{\beta^{(i)} - \beta^{(k)}} \in F
\qquad\text{for any three distinct } i,j,k \in \{1,2, 3, 4\}.
\]
So for the pair $\tau(\beta)$ defined in \eqref{eq6.1BEG}, there are only finitely many 
possibilities.  
Then Lemmas~\ref{lem6.3BEG} and \ref{lem7.2BEG} imply that there are only finitely many possibilities for such orders ${O} = \mathbb{Z}[\beta]$.

{\it Case~(ii)}. Assume  that $\alpha,\beta$ in \eqref{thepaialphabeta} satisfy alternative~(ii) of Lemma~\ref{lem9.3BEG}.  For every quadruple $(i,j,k,l)$ of distinct indices from $\{1, 2, 3, 4\}$, we have
\[
\frac{(\beta^{(i)} - \beta^{(j)})(\beta^{(k)} - \beta^{(l)})}
     {(\beta^{(i)} - \beta^{(k)})(\beta^{(j)} - \beta^{(l)})}
=
\frac{(\alpha^{(i)} - \alpha^{(j)})(\alpha^{(k)} - \alpha^{(l)})}
     {(\alpha^{(i)} - \alpha^{(k)})(\alpha^{(j)} - \alpha^{(l)})}.
\]

Since the cross ratio of any four numbers among the $\alpha^{(i)}$’s is equal 
to the cross ratio of the corresponding four numbers among the $\beta^{(i)}$’s, 
 there exists a matrix
\[
M  = 
\begin{pmatrix}
m_{1} & m_{2} \\
m_{3} & m_{4}
\end{pmatrix}
\in \textrm{GL}_2(\tilde{K})
\]
such that
\begin{equation}\label{hencehold}
\frac{m_{1}\alpha^{(i)} + m_{2}}{m_{3}\alpha^{(i)} + m_{4}} = \beta^{(i)}
\qquad\text{for } i=1,2, 3,4.
\end{equation}
We may assume that the first nonzero entry among $m_{1},m_2, m_3, m_{4}$ is $1$, then the matrix $M$ is 
uniquely determined.
 Any $\sigma \in \textrm{Gal}(\tilde{K}/k)$ permutes the two sequences 
$$
\alpha= \alpha^{(1)},\alpha^{(2)}, \alpha^{(3)}, \alpha^{(4)}\, \quad \textrm{and} \quad  \beta = \beta^{(1)},\beta^{(3)} \beta^{(3)} ,\beta^{(4)}
$$
 in the same manner, and 
\eqref{hencehold} holds with $\sigma(M)$ in place of $M$, so $\sigma(M)=M$.  
It follows that $M \in \textrm{GL}_{2}(\mathbb{Q})$. Since we assumed $\alpha$ and $\beta$ are not $\mathbb{Z}$-equivalent, we have  $m_{3}\neq 0$. Further,  by multiplying $M$ by 
a suitable integer, we may arrange that 
$$m_{1}, m_2, m_3, m_{4}\in \mathbb{Z}, \, \textrm{and} \, \gcd (m_{1},m_2, m_3,,m_{4})=1.
$$

 By Lemma \ref{lem6.4BEG}, we conclude that that ${O}=\mathbb{Z}[\alpha]=\mathbb{Z}[\beta]$ is of type~I.

{\it Case~(iii)}.  Assume  that $\alpha,\beta$ in \eqref{thepaialphabeta} satisfy alternative~(iii) of Lemma~\ref{lem9.3BEG}. Let $\varepsilon_{ij}$ be defined for this choice of $\alpha$ and $\beta$ as   in \eqref{defofepsilon}.
We define
\begin{equation}\label{defofu0}
u_{0} := \varepsilon_{12}\varepsilon_{13}\varepsilon_{14},
\end{equation}
and 
\begin{equation}\label{defofalpha0}
\alpha_{0} := \frac{1}{2}u_{0}\left(\varepsilon^{-1}_{12} + \varepsilon^{-1}_{13} 
+ \varepsilon^{-1}_{14}\right),\qquad
\beta_{0} := \frac12\left(\varepsilon_{12} + \varepsilon_{13} + \varepsilon_{14}\right).
\end{equation}

 By part~(i) of Proposition~\ref{rep9.3BEG}, $u_{0},\alpha_{0},\beta_{0}$ are invariant under $\mathrm{Gal}(\tilde{K}/K)$ and 
belong to $K$.
Moreover,
by  \eqref{epsilonijpm1} and since  we are assuming 
$\varepsilon_{ij} = - \varepsilon_{kl}$, 
for every permutation  $(i, j, k, l)$ of $(1, 2, 3, 4)$, by the definition of $u_0$ in \eqref{defofu0}, we have 
\[
\pm 1= \varepsilon_{12}\varepsilon_{34}\varepsilon_{13} \varepsilon_{24}\varepsilon_{14}\varepsilon_{23} = -u_0^2.
\]
Therefore, by part~(ii) of  Proposition~\ref{rep9.3BEG}, we obtain
\begin{equation}\label{u01}
u_0 = \pm 1.
\end{equation}

Further, using the identities in alternative~(iii) of Lemma~\ref{lem9.3BEG}, we get
\begin{equation}\label{eq9.6BEG}
\beta_{0}^{2} = \alpha_{0} + r_{0},\qquad
\alpha_{0}^{2} = u_{0}\beta_{0} + s_{0},
\end{equation}
with 
\begin{eqnarray*}
r_{0}& = &\frac14\left(\varepsilon_{12}^{2} + \varepsilon_{13}^{2} + \varepsilon_{14}^{2}\right) =  - \frac14\left(\varepsilon_{12} \varepsilon_{34}+ \varepsilon_{13}\varepsilon_{24} + \varepsilon_{14}\varepsilon_{23}\right),\\
s_{0} &= & \frac14 u_{0}^{2} \left(\varepsilon_{12}^{-2} + \varepsilon_{13}^{-2} + \varepsilon_{14}^{-2}\right) = \frac14  \left(\varepsilon_{12}^{-1} \varepsilon_{34}^{-1}+ \varepsilon_{13}^{-1} \varepsilon_{24}^{-1}+ \varepsilon_{14}^{-1}\varepsilon_{23}^{-1}\right) .
\end{eqnarray*}

Since $\varepsilon_{12} \varepsilon_{34}+ \varepsilon_{13}\varepsilon_{24} + \varepsilon_{14}\varepsilon_{23}$ and $\varepsilon_{12}^{-1} \varepsilon_{34}^{-1}+ \varepsilon_{13}^{-1} \varepsilon_{24}^{-1}+ \varepsilon_{14}^{-1}\varepsilon_{23}^{-1}$ are invariant under $\mathrm{Gal}(\tilde{K}/\mathbb{Q}) \cong A_4$, so are $r_{0}$ and $s_{0}$, and therefore
$$r_{0},s_{0}\in \mathbb{Q}.
$$
Applying   $(1 \, 2) (3\,4) \in A_4$, 
for the algebraic conjugates $\alpha_{0}^{(2)}$  abd $\beta_{0}^{(2)}$ of $\alpha_{0}$ and $\beta_{0}$, definde in \eqref{defofalpha0},  we have 
\[
\alpha_{0}^{(2)} 
= \frac{1}{2}u_{0}\left( \varepsilon_{21}^{-1} + \varepsilon_{23}^{-1} + \varepsilon_{24}^{-1}\right) 
=  \frac{1}{2}u_{0}\left(\varepsilon_{12}^{-1} - \varepsilon_{13}^{-1} - \varepsilon_{14}^{-1}\right),\quad \textrm{with}\, u_0 = \pm 1
\]
and 
\[
\beta^{(2)}_{0}
=
\frac12\left(\varepsilon_{12} + \varepsilon_{24} + \varepsilon_{23}\right)
=
\frac{1}{2}
(\varepsilon_{12} - \varepsilon_{13} - \varepsilon_{14}).
\]
By \eqref{defofu0}, and the above identities for $\alpha_{0}^{(2)}$ and $\beta_{0}^{(2)}$, we have 
\[
\frac{\alpha^{(1)}_{0} - \alpha^{(2)}_{0}}{\beta^{(1)}_{0} - \beta^{(2)}_{0}}
=  u_0 \frac{-\varepsilon_{13}^{-1} - \varepsilon_{14}^{-1}}{-\varepsilon_{13} - \varepsilon_{14}} =  u_0  \varepsilon^{-1}_{13} \varepsilon^{-1}_{14}= \varepsilon_{12}.
\]
More generally, we have
\begin{equation}\label{eq9.7BEG}
\frac{\alpha^{(i)}_{0} - \alpha^{(j)}_{0}}
{\beta^{(i)}_{0} - \beta^{(j)}_{0}}
=
\varepsilon_{ij}
\qquad
(1 \le i,j \le 4,\, i \neq j),
\end{equation}
where  $\varepsilon_{ij} = \varepsilon_{ij}(\alpha, \beta)$ is defined in \eqref{defofepsilon} for the fixed choice of monogenizers $\alpha$ and $\beta$. Since, by \eqref{epsilonijpm1}, $\varepsilon_{ij}$ are non-zero, we conclude that the four algebraic numbers $\alpha_{0} = \alpha_{0} ^{(1)}$,  $\alpha_{0} ^{(2)}$,  $\alpha_{0} ^{(3)},$ and  $\alpha_{0} ^{(4)}$  are distinct, and 
$
\mathbb{Q}(\alpha_{0})  = K$. 
Similarly, we have
$
 \mathbb{Q}(\beta_{0}) = K$.

Moreover, using \eqref{eq8.3BEG}  with  both pairs of monogenizers  
$(\alpha,\beta)$ and $(\alpha_{0},\beta_{0})$, via \eqref{eq9.7BEG} and by Lemma \ref{shabnot1},
 we obtain
\[
\frac{\beta^{(i)} - \beta^{(j)}}{\beta^{(i)}_{0} - \beta^{(j)}_{0}}
=
\frac{\beta^{(i)} - \beta^{(k)}}{\beta^{(i)}_{0} - \beta^{(k)}_{0}}
\qquad
(1 \le i,j,k \le 4,\; i,j,k \text{ distinct}).
\]
Multiplying this identity by $\varepsilon_{ij}/\varepsilon_{ik}$ yields
\[
\frac{\alpha^{(i)} - \alpha^{(j)}}{\alpha^{(i)}_{0} - \alpha^{(j)}_{0}}
=
\frac{\alpha^{(i)} - \alpha^{(k)}}{\alpha^{(i)}_{0} - \alpha^{(k)}_{0}}
\qquad
(1 \le i,j,k \le 4,\; i,j,k \text{ distinct}).
\]
This shows that
\[
\tau(\beta)=\tau(\beta_{0}) \quad \textrm{and} \quad  \tau(\alpha)=\tau(\alpha_{0}),
\]
where $\tau(\cdot)$ is defined in \eqref{eq6.1BEG}. By  part (i) of Lemma~\ref{lem6.3BEG}, there exist 
$\rho,\rho' \in \mathbb{Q}$ such that
\[
\alpha = \pm \alpha_{0} + \rho 
\qquad
\beta = \pm\beta_{0} + \rho'.
\]
Combining this with \eqref{eq9.6BEG}, we obtain
\[
\beta = \mu_{0}\alpha^{2} + \mu_{1}\alpha + \mu_{2},
\qquad
\alpha = \nu_{0}\beta^{2} + \nu_{1}\beta + \nu_{2},
\]
with $\mu_{0}, \mu_{1}, \mu_{2}, \nu_{0}, \nu_{1}, \nu_{2} \in \mathbb{Q}$ and $\mu_{0} \nu_{0} \neq 0$. However, since we assumed $\mathbb{Z}[\alpha] = \mathbb{Z}[\beta]$, we have indeed 
$$
\mu_{0}, \mu_{1}, \mu_{2}, \nu_{0}, \nu_{1}, \nu_{2} \in \mathbb{Z}.
$$
Therefore,  $O = \mathbb{Z}[\alpha] = \mathbb{Z}[\beta]$ is an order of type~II.
\end{proof}

\subsection*{ A Final Remark on Type II Relations.}
 In our proof of Lemma  \ref{newlemA4}, as well as in the original arguments in \cite{BEG} (see their proof of Theorem 3.2), it is shown that the condition $\varepsilon_{ij} = - \varepsilon_{kl}$ for every permutation $(i, j, k, l)$ of $(1, 2,3,4)$ results in having a type~II relation among two monogenizers, in the cases of $A_4$- and $S_4$-quartic number fields. One can easily see that the converse is also true; If $\alpha$ and $\beta$ have type~II relation, then by definition, there exist integers $\mu_{0}, \mu_{1}, \mu_{2}, \nu_{0}, \nu_{1}, \nu_{2} $, such that
\[
\beta = \mu_{0}\alpha^{2} + \mu_{1}\alpha + \mu_{2},
\qquad
\alpha = \nu_{0}\beta^{2} + \nu_{1}\beta + \nu_{2},
\]
with $\mu_0 \nu_0 \neq 0$.
For every $i, j \in \{1, 2, 3, 4\}$, we have
$$
\varepsilon_{ij} =  \frac{\alpha^{(i)} - \alpha^{(j)}} 
{\beta^{(i)} - \beta^{(j)}} = \nu_0 \left(\beta^{(i)} + \beta^{(j)} \right) + \nu_1.
$$
Therefore, for any permutation $(i, j, k, l)$ of $(1, 2, 3, 4)$, we have
$$
\varepsilon_{ij} + \varepsilon_{kl}  =\nu  \in \mathbb{Z},
$$
with
$$
\nu= \nu_0 \left(\beta_i + \beta_j + \beta_k + \beta_l\right) + 2 \nu_1.
$$
Similarly, we have 
$$
\left(\varepsilon_{ij}\right)^{-1} + \left(\varepsilon_{kl}\right)^{-1} = \mu \in \mathbb{Z},
$$
with 
$$
\mu = \mu_0 \left(\alpha_i + \alpha_j + \alpha_k + \alpha_l\right) + 2 \mu_1.
$$
Therefore,
$$
1 = \varepsilon_{ij}  \left(\varepsilon_{ij}\right)^{-1}  = \left( \nu -  \varepsilon_{kl} \right) \left( \mu -  \left(\varepsilon_{kl}\right)^{-1} \right) = \nu \mu + 1-\frac{\nu}{ \varepsilon_{kl}}-\mu \varepsilon_{kl}.
$$
In other words,  for every distinct $k, l \in \{1, 2, 3, 4\}$, the algebraic integer  $\varepsilon_{kl} \neq 0$ satisfies the quadratic equation
$$
- \mu T^2 +\mu \nu  T - \nu =0,
$$
where $T$ is the variable.
By Lemma \ref{shabnot1}, $\varepsilon_{k i}$, $\varepsilon_{k j}$, and $\varepsilon_{k l}$ are distict non-zero algebraic integers. 
 We conclude that $\mu= \nu = 0$, whereby
$$
\varepsilon_{ij} = - \varepsilon_{kl},
$$
for any permutation $(i, j, k, l)$ of $(1, 2, 3, 4)$.

\section*{Acknowledgements} 
The first author gratefully acknowledges the support and hospitality of \emph{Lodha Mathematical Sciences Institute, Mumbai} and thanks the organizers of the Arithmetic Statistics program during the fall semester 2025, held at LMSI, where major progress was made towards the completion of this project.  This work was partially supported by the NSF award DMS-2327098.



\begin{thebibliography}{99}

\bibitem{AkhEss} S.~Akhtari, Quartic index form equations and monogenizations of quartic orders, {\it Essential Number Theory}, vol. 1 no.1 (2022)



\bibitem{Ben}
M.~A.~Bennett, On the representation of unity by binary cubic forms, {\it Trans.\ Amer.\ 
Math.\ Soc.} {\bf 353} (2001), 1507--1534.



\bibitem{BEG} {\sc B\'{e}rczes, Evertse, and Gy\H{o}ry}, Multiply monogenic orders, {\it Ann. Sc. Norm. Super. Pisa Cl. Sci. } $5$ Vol. XII (2013), 467--497.




\bibitem{Bha-notes}  M.~Bhargava, On the number of monogenizations of a quartic order (with an appendix by Shabnam Akhtari), {\it Publicationes Mathematicae Debrecen},  {\bf 100} (2022), 513--531.



\bibitem{GaussI}
M.~Bhargava, Higher composition laws I: A new view on Gauss composition, and
quadratic generalizations, {\it Ann. of Math.} {\bf 159} (2004), 217--250.


\bibitem{Bha04}  M.~Bhargava,  Higher composition laws III: The parametrization of quartic rings,  {\it Annals of Mathematics}, {\bf 159} (2004), 1329--1360




\bibitem{BrMa} M.~Brennan, D.~Machale, Variations on a theme: $A_4$ definitely has no subgroup of order six!, Math. Mag.
{\bf 73} (2000), 36--40.


\bibitem{Ded} R.~Dedekind, \"Uber die Zusammenhang zwischen der Theorie der Ideale und der Theorie
der h\"oheren Kongruenzen, {\it Abh. K\"onig. Ges. Wissen. G\"ottingen} {\bf 23} (1878), 1--23.




\bibitem{DF} D.~S.~Dummit and R.~M.~Foote, {\it Abstract Algebra}, 3rd ed., Wiley, 2004.


\bibitem{Eve11} J.~H.~Evertse, A  survey on monogenic orders,   {\it  Publ. Math. Debrecen}  $79$ (3) (2011), 411--422.



\bibitem{EG85} J.~H.~Evertse and K. Gy\H{o}ry, On unit equations and decomposable form equations, {\it J.Reine Angew. Math.} {\bf 358} (1985), 6--19.

\bibitem{EGST} J.~H.~Evertse, K. Gy\H{o}ry, C.~L. Stewart, and R. Tijdeman,
{\it On $S$-unit equations in two unknowns},
Invent. Math. \textbf{92} (1988), 461--477.


 \bibitem{EGbook}   J.~H.~Evertse and K.~Gy\H{o}ry,    Discriminant Equations in Diophantine Number Theory, vol. 32


\bibitem{thebook}  I.~Ga\'al,  Diophantine Equations and Power Integral Bases, Theory and Algorithms.
2nd edition, Birkh\"auser, Boston, 2019.
 
 \bibitem{GaPePo} I.~Ga\'al, A.~Peth\H{o} and M.~Pohst, Simultaneous Representation of Integers by a pair of quartic forms- with an application to index form equations in quartic number fields, {\it  J.
Number Theory} {\bf 57} (1996), 90--104.

 

\bibitem{Gyo76} K.~ Gy\H{o}ry, Sur les polyn\^omes \`a coefficients entiers et de discriminant donn\'e III, {\it Publ. Math. Debrecen}, {\bf 23} (1976),  141--165.

\bibitem{Has}H.~Hasse, Number Theory, Akademie-Verlag (1979).

\bibitem{Hen} K.~Hensel, Theorie der algebraischen Zahlen, Teubner Verlag, Leipzig and Berlin (1908). 




\bibitem{KW89} L.~C.~Kappe and B.~Warren, 
An elementary test for the Galois group of a quartic polynomial,
{\it Amer. Math. Monthly} \textbf{96} (1989), 133--137.




\bibitem{Thu}  A.~Thue, \"{U}ber Ann\"{a}herungswerte algebraischer Zahlen, {\it J. reine
angew. Math.} {\bf 135} (1909), 284--305.




 \end{thebibliography}
\end{document}